\documentclass{amsart}
\usepackage{amssymb}
\usepackage{graphicx}
\def\({\left(}
\def\){\right)}

\newtheorem{lema}{Lemma}[section]

\newtheorem*{teorema*}{Theorem}

\newtheorem{lemma}{Lemma}[section]

\newtheorem{theorem}[lema]{Theorem}

\newtheorem{definition}[lema]{Definition}

  {\par \hfill \fbox{}}
  {\par \hfill \fbox{}}
  {\par \hfill }
  {\par \hfill }

\def\beq{\begin{equation}}
\def\eeq{\end{equation}}

\def\epsilon{\varepsilon}

\begin{document}

\title{Li-Yorke chaos for Composition operators on Orlicz-Lorentz space }
\author{Rajat Singh}
\address{Department of Mathematics, University of Jammu,
Jammu 180006, INDIA.} \email{rajat.singh.rs634@gmail.com}

\author{Aditi Sharma}
\address{Department of Mathematics, University of Jammu,
Jammu 180006, INDIA.} \email{aditi.sharmao@gmail.com}

\author{Romesh Kumar}
\address{Department of Mathematics, University of Jammu,
Jammu 180006, INDIA.} 
\email{romeshmath@gmail.com}

\thanks{The first author is supported by UGC under the scheme of JRF}
\subjclass{Primary }
\date{September 2022}

\begin{abstract} 
In this paper, we study the Li-Yorke chaotic composition operators on Orlicz-Lorentz spaces.  In fact necessary and sufficient conditions are given for Li-Yorke chaotic composition operator $C_{\tau}$ on $\mathbb{L}^{\varphi,h}(\mu)$. Further, we present the equivalent conditions for $C_{\tau}$ to be Li-Yorke chaotic. We extends the result of \cite{NC19} to Orlicz-Lorentz spaces.
\end{abstract}

\maketitle

\section{Introduction and Preliminaries}
Over the last two decades various authors have explored chaotic operators intensively. The notion of ``Chaos" was first introduced into mathematical literature in the context of interval map by Li and Yorke \cite{Li} and became most famous. Godefroy and Shapiro \cite{god} were the first to introduce chaos into linear dynamics by using Devaney's notion of chaos. Other terms related to chaos include Li-Yorke chaos, distributional chaos, specification property, etc. There are several fascinating papers of Li-Yorke chaos and distributional chaos for linear operator see (\cite{NC13},\cite{NC15},\cite{NC19},\cite{Li}) and reference therein. Further in \cite{Bam}, the Li-Yorke chaos for composition operator on Orlicz space were studied. Since the Orlicz-Lorentz spaces offer a common generalization of Orlicz spaces and Lorentz spaces, so it is natural to expand the study of composition operators to a more general class.\\
The paper is structured as follows: Section 1 is introductory and we cite certain definitions and results which will be used throughout this paper. In Section 2, we explore the Li-Yorke chaotic composition operators on Orlicz-Lorentz spaces.\\
Now, we recall some basic facts about the Orlicz-Lorentz space which will be useful throughout this paper. For more details on Orlicz-Lorentz space we refer to \cite{BS}, \cite{Cui}, \cite{LOR} and \cite{Rao}.
Let $(X,\mathbb{A},\mu)$ be a measure space with positive measure and $\mathbb{L}^{0}$ represent the space of all equivalence classes of measurable functions on $X$ which are identified as $\mu$-a.e.
We now define the distribution function $\mu_{g}$ of $g\in L^{0}$ on $(0,\infty)$ as
$$\mu_{g}(\lambda)=\mu\{x \in X:|g(x)|>\lambda\},$$
and the non-increasing rearrangement of $g$ on $(0,\infty)$ is defined as
$$g^{*}(t)=\inf\{\lambda>0:\mu_{g}(\lambda)\leq t\}=\sup\{\lambda>0:\mu_{g}(\lambda)>t\}.$$
A function $\varphi:[0,\infty) \to [0,\infty]$ is an Orlicz function if it is convex function with $\varphi(0)=0$ and $\varphi(s) \to \infty~\mbox{as}~s \to \infty$ such that $\varphi(s)<\infty$ for some $0<s<\infty$. Let $K=[0,\mu(X)]$. Let $h:K \to (0,\infty)$ be locally integrable and non-increasing function with respect to Lebesgue measure which is known as weight function. 
For a given $\varphi$ and $h$, the space
$$\mathbb{L}^{\varphi,h}(\mu)=\{ g\in \mathbb{L}^{0}:I_{\varphi,h}(\lambda g)<\infty~ \mbox{for some}~\lambda>0\},$$
where $$I_{\varphi,h}(\lambda g)=\int_{K}\varphi(\lambda g^{*}(t))h(t)dt.$$ is known as Orlicz-Lorentz space.
Now, we define the norm of $g$ as
 $$||g||_{\varphi,h}=\inf\{\lambda>0:I_{\varphi,h}(g/\lambda)\leq 1\}.$$ 
The space $\mathbb{L}^{\varphi,h}(\mu)$ with the above norm is a Banach space. Moreover, if $A\in \mathbb{A}$ with $0<\mu(A)<\infty$, then $||\chi_{A}||_{\varphi,h}=\frac{1}{\varphi^{-1}\(\frac{1}{\int_{0}^{\mu(\tau^{-n}(A))}h(t)dt}\)}$. If $h=1$, then $\mathbb{L}^{\varphi,h}$ represents the Orlicz space $\mathbb{L}^{\varphi}(\mu)$ and when $\varphi(t)=t$, then it represents the Lorentz space $\mathbb{L}^{w}(\mu)$ see \cite{ROL}.

First of all, we now state some growth conditions on the $\varphi$, the Orlicz function $\varphi$ satisfies the $\Delta_{2}$-condition(for large s), if there is a positive constant $M$(a positive constant M and $s_{o}>0$ with $\varphi(s_{o})<\infty$) such that
\begin{equation}\label{eqdelta}
	\varphi(2s)\leq M\varphi(s), ~\mbox{for all}~ s>0.
\end{equation}
We suppose that $\varphi$ is left continuous at $b_{\phi}$, where
$$b_{\phi}=\sup\{s>o:\phi(s)<\infty\}.$$
Also, define $$a_{\phi}=\inf\{s>0:\phi(s)>0\}.$$
Throughout this paper, we take $(X,\mathbb{A},\mu)$ to be measure space and let $\tau:X \to X$ be non-singular measurable transformation. Then $\tau$ induces a operator $C_{\tau}$ from $\mathbb{L}^{\varphi,h}$ into $L^{0}(X)$ defined as 
$$C_{\tau}g=g \circ \tau,~\forall~g \in \mathbb{L}^{\varphi}(\mu) $$
The non-singularity of $\tau$ assures that $C_{\tau}$ is well defined. Now, we state the lemma which will be used in the proof of Theorem\ref{thrm2.3}.
\begin{lemma}\label{lem1} 
Let $ \tau$ be measurable non-singular transformation and $\varphi$ be an Orlicz function which satisfies the $\Delta_{2}$-condition for all $s>0$. Then the following are equivalent:
\begin{itemize}
\item[(a)] There exist some $k>0$ and $A\in \mathbb{A}$,
      $$\frac{1}{\varphi^{-1}\(\frac{1}{\int_{0}^{\mu(\tau^{-n}(A))}h(t)dt}\)} \leq k \frac{1}{\varphi^{-1}\(\frac{1}{\int_{0}^{\mu(A)}h(t)dt}\)}.$$
\item[(b)] There exist some $k>0$ and $A\in \mathbb{A}$,
   $$\int_{0}^{\mu(\tau^{-n}(A))}h(t)dt \leq k\int_{0}^{\mu(A)}h(t)dt.$$
	\end{itemize}
\end{lemma}

The proof of the lemma follows directly from the \cite[Theorem 2.1]{ROL}, by using the fact that $ \varphi$ satisfies the $\Delta_{2}$-condition for large $s>0$.
\begin{definition}\cite[Page 1]{Berm}
	A continuous map $g:(M,d) \to (M,d)$  is said to be Li-Yorke chaotic if there exists an uncountable set $S\subset M$  such that each pair of distinct points $p,q \in S$ is a Li-Yorke pair for $g$ i.e.,
	$$\displaystyle\lim_{n\to \infty}\inf d(g^{n}(p), g^{n}(q)) = 0 ~\mbox{and}~ \displaystyle\lim_{n\to \infty}\sup d(g^{n}(p), g^{n}(q)) > 0 ,$$
	where $(M,d)$ is a metric space.
\end{definition}
We say that $g$ is densely (generically) Li-Yorke chaotic whenever $S$ can be chosen to be dense (residual) in $M$. 
\begin{definition}\cite[Page 47]{B2}
	\begin{itemize}
		\item[(a)] If  $T$ is a linear operator and a vector $z \in X$, then we say that $z$ is an irregular vector for $T$ if 
		$$\displaystyle\lim_{n \to \infty} \inf||T^{n}z||=0~\mbox{and}~\displaystyle\lim_{n \to \infty} \sup||T^{n}z||=\infty.$$
		\item[(b)] If $T$ is a linear operator and a vector $z \in X$, then we say that $z$ is semi-irregular vector for $T$ if 
		$$\displaystyle\lim_{n \to \infty} \inf||T^{n}z||=0~\mbox{and}~\displaystyle\lim_{n \to \infty} \sup||T^{n}z||>0.$$ 
	\end{itemize}
\end{definition}

Following result gives the equivalent condition for any continuous linear operator $T$ on any Banach space to be Li-Yorke. 
\begin{theorem}\cite[Theorem 9]{NC15}\label{prop}
	Let $X$ be a Banach space and $T:X \to X$ be a bounded linear operator. Then the following are equivalent
	\begin{itemize}
		\item[(i)] $T$ is Li-Yorke chaotic.
		\item[(ii)] $T$ admits a semi-irregular vector.
		\item[(iii)] $T$ admits irregular vector. 
	\end{itemize} 
\end{theorem}

\section{Li-Yorke chaotic composition operators on Orlicz-Lorentz Space}
In this section, we discuss Li-Yorke chaotic composition operators on Orlicz-Lorentz spaces.

\begin{theorem}\label{nasc}
	Let $C_{\tau}$ be a continuous linear operator on $\mathbb{L}^{\varphi, h}(\mu)$. Then the composition operator on $\mathbb{L}^{\varphi, h}(\mu)$ is Li-Yorke chaotic iff there is an increasing sequence of positive integers $\{\gamma_{k}\}_{k \in \mathbb{N}}$ and a countable family of non-empty measurable sets $\{A_{i}\}_{i\in I}$ with $0<\mu(A_{i})<\infty$ such that
	\begin{itemize}
	\item[(a)] $\displaystyle\lim_{k \to \infty}\varphi^{-1}\(\frac{1}{\int_{0}^{\mu(\tau^{-\gamma_{k}}(A_{i}))} h(t)dt}\)=\infty,~\forall i \in I.$
	\item[(b)] $\displaystyle\sup \left\{ \frac{\varphi^{-1}\(\frac{1}{\int_{0}^{\mu(A_{i})} h(t)dt}\)}{\varphi^{-1}\(\frac{1}{\int_{0}^{\mu(\tau^{-n}(A_{i}))} h(t)dt}\)}:i\in I, n\in \mathbb{N}\right\}=\infty.$
	\end{itemize}
\end{theorem}

\begin{proof}
First, we suppose that $C_{\tau}$ is a continuous Li-Yorke chaotic composition operator on Orlicz-Lorentz space $\mathbb{L}^{\varphi, h}(\mu)$. Then there is an irregular vector $g\in \mathbb{L}^{\varphi, h}(\mu)$ for $C_{\tau}$. Let $A_{i}=\{ x \in X: 3^{i-1}\leq |g(x)|<3^{i}\},~i\in \mathbb{Z}.$ Then clearly, $A_{i}\in \mathbb{A},~\forall i\in \mathbb{Z}$ and suppose $I=\{i\in \mathbb{Z}:\mu(A_{i})>0\}.$ As $g\in \mathbb{L}^{\varphi,h}(\mu)$, so there is some $\lambda>0$ such that
$$\int_{K}\varphi(\lambda f^{*}(t))h(t)dt<\infty.$$
Thus for all $i \in I$, we get
\begin{eqnarray*}
\varphi(3^{i-1}\lambda)\mu(A_{i}) &=& \int_{A_{i}}\varphi(\lambda|g(t)|)d\mu(t) \\
                            &=& \int_{K}\varphi(\lambda g^{*}(t)) h(t)d(t) \\
                             &<& \infty.
\end{eqnarray*}
This implies that $0<\mu(A_{i})<\infty,\forall~i\in I$. Since $g$ is an irregular vector, so there is an increasing sequence $\{\gamma_{k}\}_{k\in \mathbb{N}}$ of positive integers such that 
\begin{equation}\label{eq1}
\displaystyle\lim_{k \to \infty}||C_{\tau}^{\gamma_{k}}g||_{\varphi,h}=0.
\end{equation}
Now, by using the fact that $h$ is decreasing, we have 
\begin{eqnarray*}
\int_{K}\varphi\(\frac{(3^{i-1}\chi_{\tau^{-\gamma_{k}(A_{i})}})^{*}(t)}{M||C_{\tau}^{\gamma_{k}}g||_{\varphi,h}}\)h(t)dt 
        &=& \int_{\tau^{-\gamma_{k}}(A_{i})}\varphi\(\frac{(3^{i-1}\chi_{A_{i}}\circ \tau^{\gamma_{k}})(t)}{M||C_{\tau}^{\gamma_{k}}g||_{\varphi,h}}\)h(t)d\mu(t) \\
        &=& \int_{\tau^{-\gamma_{k}}(A_{i})}\varphi\(\frac{|g\circ \tau^{\gamma_{k}}(t)|}{M||C_{\tau}^{\gamma_{k}}g||_{\varphi,h}}\)h(t)d\mu(t) \\
        &=& \int_{0}^{\mu(\tau^{-\gamma_{k}}(A_{i}))}\varphi\(\frac{(g\circ \tau^{\gamma_{k}})^{*}(t/M)}{M||C_{\tau}^{\gamma_{k}}g||_{\varphi,h}}\)h(t)dt \\
				&\leq & \int_{K} \varphi\(\frac{(g \circ\tau^{\gamma_{k}})^{*}(u)}{||C_{\tau}^{\gamma_{k}}g||_{\varphi,h}}\)h(u)du \\
				&\leq& 1.
\end{eqnarray*}		
This means that	$||\chi_{\tau^{-\gamma_{k}}(A_{i})}||_{\varphi,h}\leq M||C_{\tau^{\gamma_{k}}}g||_{\varphi,h}$. By using the Equation(\ref{eq1}), we get
\begin{eqnarray*}
\displaystyle\lim_{k \to \infty}||\chi_{\tau^{-\gamma_{k}}(A_{i})}||_{\varphi,h} &=& 0 \\
\displaystyle\lim_{k \to \infty}\varphi^{-1}\(\frac{1}{\int_{0}^{\mu(\tau^{-\gamma_{k}}(A_{i}))}h(t)dt}\) &=& \infty,~i\in \mathbb{Z}.
\end{eqnarray*}
Thus (a) holds. Now, for the second suppose on the contrary that (b) will not be holds. Then for some $r>0$, we have
$$\varphi^{-1}\(\frac{1}{\int_{0}^{\mu(A_{i})} h(t)dt}\) \leq r \varphi^{-1}\(\frac{1}{\int_{0}^{\mu(\tau^{-n}(A_{i}))} h(t)dt}\),~\forall~i\in I,~n\in \mathbb{N}.$$
By using the Lemma \ref{lem1}, we see that 
    $$\int_{0}^{\mu(\tau^{-n}(A_{i}))} h(t)dt \leq r \int_{0}^{\mu(A_{i})}h(t)dt,~\mbox{ for some}~ r>0.$$
Now, for every $n \in \mathbb{N}$,
\begin{eqnarray*}
\int_{K}\varphi\(\frac{(g \circ \tau^{n})^{*}(t)}{3r||g||_{\varphi,h}}\)h(t)dt &=&\sum_{i\in I}\int_{0}^{\mu(\tau^{-n}(A_{i}))}\varphi\(\frac{(g \circ \tau^{n})^{*}(t)}{3r||g||_{\varphi,h}}\)h(t)dt \\
            &=& \sum_{i\in I}\int_{\tau^{-n}(A_{i})}\varphi\(\frac{(g \circ \tau^{n})(t)}{3r||g||_{\varphi,h}}\)h(t)d\mu(t) \\
            &<& \sum_{i\in I}\varphi\(\frac{3^{i}}{3r||g||_{\varphi,h}}\)\int_{\tau^{-n}(A_{i})}h(t)d\mu(t) \\
						&=& \sum_{i\in I}\varphi\(\frac{3^{i-1}}{r||g||_{\varphi,h}}\)\int_{0}^{\mu(\tau^{-n}(A_{i}))}h(t)dt \\
						&<& \sum_{i\in I}\varphi\(\frac{3^{i-1}}{||g||_{\varphi,h}}\)\int_{0}^{\mu(A_{i})}h(t)dt \\
						&=& \sum_{i\in I}\int_{0}^{\mu(A_{i})} \varphi\(\frac{|g(t)|}{||g||_{\varphi,h}}\)h(t)dt \\
            &=& \sum_{i\in I}\int_{A_{i}} \varphi\(\frac{|g(t)|}{||g||_{\varphi,h}}\)h(t)d\mu(t) \\
            &=& \sum_{i\in I}\int_{0}^{\mu(A_{i})} \varphi\(\frac{g^{*}(t)}{||g||_{\varphi,h}}\)h(t)d(t) \\
						&=& \int_{K} \varphi\(\frac{g^{*}(t)}{||g||_{\varphi,h}}\)h(t)d(t) \\
            &\leq& 1.
\end{eqnarray*}
i.e., $||C_{\tau^{n}}g||_{\varphi,h}\leq 3r||f||_{\varphi,h}$, which is a contradiction to the fact that $g$ is irregular and so the condition (b) holds.\\
Conversely, suppose that the condition (a) and (b) holds for measurable non-empty countable family $\{A_{i}\}_{i\in \mathbb{N}}$ with $0<\mu(A_{i})<\infty$. Let $S$ be a linear span of $\overline{\{\chi_{A_{i}}:i\in I\}}$ in Orlicz-Lorentz space. Then the orbit of $C_{\tau}$ is 
 $$\mathcal{O}_{C_{\tau}g}=\{g,C_{\tau}g,...\}.$$
By using the conditions (a), we see that the set $S_{1}=\{g\in S: \mathcal{O}_{C_{\tau}g}~ \mbox{has a subsequence converging to 0}\}$ is dense in S.\\
 Let $g_{i}=\varphi^{-1}\(\frac{1}{\int_{0}^{\mu(A_{i})}h(t)dt}\).\chi_{A_{i}} \in S$. Then clearly, $||g_{i}||_{\varphi,h}=1$ and $||C_{\tau}^{n}g_{i}||_{\varphi,h}=\frac{\varphi^{-1}\(\frac{1}{\int_{0}^{\mu(A_{i})} h(t)dt}\)}{\varphi^{-1}\(\frac{1}{\int_{0}^{\mu(\tau^{-n}(A_{i}))} h(t)dt}\)}$ and so from (b), we get $\displaystyle\sup_{n\in \mathbb{N}}||C_{\tau}^{n}/_{S}||_{\varphi, h}=\infty$. Now, by the Banach-Steinhaus theorem, the set $S_{2}=\{g \in S:\mathcal{O}_{C_{\tau}g} ~\mbox{is unbounded}\}$ is residual in S and hence $C_{\tau}$ is Li-Yorke chaotic as $g \in S_{1}\cap S_{2}$.
\end{proof}

\begin{theorem}
Suppose that the non-singular measurable transformation $\tau$ is injective. The composition operator $C_{\tau}$ is Li-Yorke chaotic if there is a measurable set $A$ with $0<\mu(A)<\infty$ such that
\begin{itemize} 
\item[(i)] $\displaystyle\lim_{n \to \infty} \sup\varphi^{-1}\(\frac{1}{\int_{0}^{\mu(\tau^{-n}(A_{i}))} h(t)dt}\)=\infty,$
\item[(ii)] $\sup\left\{\frac{\varphi^{-1}\(\frac{1}{\int_{0}^{\mu(\tau^{q}(A))}h(t)dt}\)}{\varphi^{-1}\(\frac{1}{\int_{0}^{\mu(\tau^{p}(A))}h(t)dt}\)}:p,q\in I,p<q\right\}=\infty,$
\end{itemize}
where $I=\{q\in \mathbb{Z}:0<\mu(\tau^{q}(A))<\infty\}$.
\end{theorem}
\begin{proof}
Since $\tau$ is injective. Taking $A_{i}=\tau^{i}(A),~i \in \mathbb{Z}$. Then by using the (i), it follows that the condition (a) of Theorem \ref{nasc} will hold. Now for $p,q \in I$, with $p<q,$ we have
$$A_{p}=\tau^{p}(A)=\tau^{p-q}(\tau^{q}(A))=\tau^{p-q}(A_{q}).$$
Hence,
\begin{eqnarray*}
\sup\left\{\frac{\varphi^{-1}\(\frac{1}{\int_{0}^{\mu(\tau^{q}(A))}h(t)dt}\)}{\varphi^{-1}\(\frac{1}{\int_{0}^{\mu(\tau^{p}(A))}h(t)dt}\)}:p,q\in I,p<q\right\} 
              &=& \sup\left\{\frac{\varphi^{-1}\(\frac{1}{\int_{0}^{\mu(A_{q})}h(t)dt}\)}{\varphi^{-1}\(\frac{1}{\int_{0}^{\mu(\tau^{p-q}(A_{q}))}h(t)dt}\)}:p,q\in I,p<q\right\}\\
              &=& \sup\left\{\frac{\varphi^{-1}\(\frac{1}{\int_{0}^{\mu(A_{i})}h(t)dt}\)}{\varphi^{-1}\(\frac{1}{\int_{0}^{\mu(\tau^{p}(A_{i}))}h(t)dt}\)}:p\in I, i\in I\right\}.
\end{eqnarray*}							
The condition (ii) implies that (b) of Theorem \ref{nasc} holds and thus $C_{\tau}$ is Li-Yorke chaotic.
\end{proof}

We now provided some necessary and sufficient conditions for the Li-Yorke chaotic composition operator $C_{\tau}$ on $L^{\varphi,h}$. 
\begin{theorem}\label{thrm2.3}
Let $\tau$ be injective and $\varphi$ satisfies the $\Delta_{2}$ condition. Then the following are equivalent:
\begin{itemize}
\item[(a)] The composition operator $C_{\tau}$ on $\mathbb{L}^{\varphi,h}(\mu)$ is Li-Yorke chaotic.
\item[(b)] There exists measurable function $0\neq g\in \mathbb{L}^{\varphi,h}(\mu)$ such that 
        $$\displaystyle\lim_{n \to \infty}\inf||C_{\tau}^{n}g||_{\varphi,h}=0.$$
\item[(c)] There exists $A\in \mathbb{A}$ with $0<\mu(A)<\infty$ such that
         $$\displaystyle\lim_{n \to \infty}\sup\varphi^{-1}\(\frac{1}{\int_{0}^{\mu(\tau^{-n}(A))}h(t)dt}\)=\infty.$$
\item[(d)] There exists a measurable set $A\in \mathbb{A}$ with $0<\mu(A)<\infty$ such that
        $$\displaystyle\lim_{n \to \infty}\sup\varphi^{-1}\(\frac{1}{\int_{0}^{\mu(\tau^{n}(A))}h(t)dt}\)=\infty.$$
\item[(e)] There exists a measurable set $A\in \mathbb{A}$ with $0<\mu(A)<\infty$ such that				
$$\displaystyle\lim_{n \to \infty}\sup\varphi^{-1}\(\frac{1}{\int_{0}^{\mu(\tau^{-n}(A))}h(t)dt}\)=\infty~\mbox{and}~\displaystyle\lim_{n \to \infty}\sup\varphi^{-1}\(\frac{1}{\int_{0}^{\mu(\tau^{n}(A))}h(t)dt}\)=\infty.$$
\item[(f)] There exists a measurable set $A\in \mathbb{A}$ with $0<\mu(A)<\infty$ such that
$$\displaystyle\lim_{n \to \infty}\inf\varphi^{-1}(\frac{1}{\int_{0}^{\mu(\tau^{-n}(A))}h(t)dt})>0~\mbox{and}~\displaystyle\lim_{n \to \infty}\sup\varphi^{-1}\(\frac{1}{\int_{0}^{\mu(\tau^{-n}(A))}h(t)dt}\)=\infty.$$
\item[(g)] There is a measurable set $A$ such that the characteristic function is a semi-irregular vector for $C_{\tau}$ on $\mathbb{L}^{\varphi,h}$.
\end{itemize}
\end{theorem}

\begin{proof} 
First, suppose that $C_{\tau}$ is Li-Yorke chaotic. Then there exists $g \in \mathbb{L}^{\varphi,h}(\mu)$ is semi-irregular function such that 
$$\displaystyle\lim_{n \to \infty}\inf ||C_{\tau}^{n}g||_{\varphi,h}=0~\mbox{and}~\displaystyle\lim_{n \to \infty}\inf ||C_{\tau}^{n}g||_{\varphi,h}>0.$$
From the earlier inequality, it follows that $g\neq 0$ and so the condition (b) holds.
Now, we prove the relation $(b)\Leftrightarrow (c)$. First assume that the condition (b) holds. That is there is a some function $0 \neq g \in \mathbb{L}^{\varphi,h}$ such that 
$\displaystyle\lim_{n \to \infty}\inf ||C_{\tau}^{n}g||_{\varphi,h}=0$. Since $g \neq 0$, there is positive constant $\epsilon$ such that $\mu(A=\{x\in X:|g(x)|>\epsilon\})>0$. Further,
\begin{eqnarray*}
\int_{K}\varphi\(\frac{(\epsilon \chi_{\tau^{-n}(A)})^{*}}{||C_{\tau}^{n}g||_{\varphi,h}}\)h(t)dt &=& \int_{0}^{\mu(\tau^{-n}(A))}\varphi\(\frac{(\epsilon \chi_{\tau^{-n}(A)})^{*}}{||C_{\tau}^{n}g||_{\varphi,h}}\)h(t)dt \\
          &=& \int_{\tau^{-n}(A)}\varphi\(\frac{\epsilon \chi_{\tau^{-n}(A)}}{||C_{\tau}^{n}g||_{\varphi,h}}\)h(t)d\mu(t) \\
          &\leq& \int_{\tau^{-n}(A)}\varphi\(\frac{|g\circ \tau^{n}|}{||C_{\tau}^{n}g||_{\varphi,h}}\)h(t)d\mu(t) \\
					&=& \int_{0}^{\mu(\tau^{-n}(A))}\varphi\(\frac{(g\circ \tau^{n})^{*}(t)}{||C_{\tau}^{n}g||_{\varphi,h}}\)h(t)dt \\
          &\leq& 1
\end{eqnarray*}
i.e., \begin{eqnarray*}
\epsilon ||\chi_{\tau^{-n}(A)}||_{\varphi,h} &\leq& ||C_{\tau}^{n}g||_{\varphi,h} \\
\epsilon \frac{1}{\varphi^{-1}\(\frac{1}{\int_{0}^{\mu(\tau^{-n}(A))} h(t)dt}\)} &\leq& ||C_{\tau}^{n}g||_{\varphi,h} 
\end{eqnarray*}
and by using the given hypothesis, we get
$$\epsilon \displaystyle\lim_{n \to \infty}\inf \frac{1}{\varphi^{-1}\(\frac{1}{\int_{0}^{\mu(\tau^{-n}(A))} h(t)dt}\)} \leq \displaystyle\lim_{n \to \infty}\inf ||C_{\tau}^{n}g||_{\varphi,h} = 0.$$
This implies that $$ \displaystyle\lim_{n \to \infty}\sup \varphi^{-1}\(\frac{1}{\int_{0}^{\mu(\tau^{-n}(A))} h(t)dt}\)=\infty.$$
In order to prove the reverse inclusion, we take $A\in \mathbb{A}$ with $0<\mu(A)<\infty$ such that $\displaystyle\lim_{n \to \infty}\sup\varphi^{-1}\( \frac{1}{\int_{0}^{\mu(\tau^{-n}(A))}h(t)dt}\)=\infty.$ Then by taking $g=\chi_{A},$ we see that 
\begin{eqnarray*}
\displaystyle\lim_{n \to \infty} \inf ||C_{\tau}^{n} \circ \chi_{A}||_{\varphi,h} &=& \displaystyle\lim_{n \to \infty} \inf || \chi_{\tau^{-n}(A)}||_{\varphi,h}\\ 
                                            &=& \displaystyle\lim_{n \to \infty} \inf \frac{1}{\varphi^{-1}\(\frac{1}{\int_{0}^{\mu(\tau^{-n}(A))}h(t)dt}\)}\\
                                            &=& \frac{1}{\displaystyle\lim_{n \to \infty} \sup \varphi^{-1}\(\frac{1}{\int_{0}^{\mu(\tau^{-n}(A))}h(t)dt}\)} \\
																						&=& 0.
\end{eqnarray*}
This implies $(c)\implies (b)$.\\
$(f)\Leftrightarrow (g)$:- By putting $g =\chi_{A}$ for all $A \in \mathbb{A}$ with $0<\mu(A)<\infty$. Since
\begin{eqnarray*}
||C_{\tau}^{n}g||_{\varphi,h} &=& ||C_{\tau}^{n} \chi_{A}||_{\varphi,h}\\
                              &=& ||\chi_{\tau^{-n}(A)}||_{\varphi,h} \\
															&=& \frac{1}{\varphi^{-1}\(\int_{0}^{\mu(\tau^{-n}(A))}h(t)dt\)},
\end{eqnarray*}
which implies that $(f)$ and $(g)$ are equivalent property.\\
 The condition (e) $\implies$ (c), (e) $\implies$ (d) are directly and (g) $\implies$ (a) follows from the Theorem \ref{prop}. Now, in order to prove the condition (e) $\implies $ (f), we assume that the $\tau$ to be injective and the Orlicz function $\varphi$ satisfy the $\Delta_{2}$ conditions for large $s>0$. Suppose there is $A\in \mathbb{A}$ with $0<\mu(A)<\infty$ such that 
$$ \displaystyle\lim_{n \to \infty}\sup \varphi^{-1}\(\frac{1}{\int_{0}^{\mu(\tau^{-n}(A))} h(t)dt}\)=\infty $$
and $$\displaystyle\lim_{n \to \infty}\sup \varphi^{-1}\(\frac{1}{\int_{0}^{\mu(\tau^{n}(A))} h(t)dt}\)=\infty.$$
As $\varphi$ is continuous, increasing and satisfy the $\Delta_{2}$ condition, we get
$$ \displaystyle\lim_{n \to \infty}\sup \(\frac{1}{\int_{0}^{\mu(\tau^{-n}(A))} h(t)dt}\)=\infty ~\mbox{and}~
\displaystyle\lim_{n \to \infty}\sup \(\frac{1}{\int_{0}^{\mu(\tau^{n}(A))} h(t)dt}\)=\infty.$$
 and so $\displaystyle\lim_{n \to \infty}\inf \int_{0}^{\mu(\tau^{-n}(A))}h(t)dt =0~\mbox{and}~\displaystyle\lim_{n \to \infty}\inf \int_{0}^{\mu(\tau^{n}(A))}h(t)dt =0.$\\
Thus by using the Lemma \ref{lem1} and \cite[corollary 1.3]{NC19}, we have 
 $$\displaystyle\lim_{n \to \infty}\inf \int_{0}^{\mu(\tau^{-n}(A))}h(t)dt =0~\mbox{and}~\displaystyle\lim_{n \to \infty}\sup \int_{0}^{\mu(\tau^{-n}(A))}h(t)dt >0,$$
and so it follows that 
$$\displaystyle\lim_{n \to \infty}\inf \varphi^{-1}\(\frac{1}{\int_{0}^{\mu(\tau^{-n}(A))}h(t)dt}\) >0~\mbox{and}~\displaystyle\lim_{n \to \infty}\sup \varphi^{-1}\(\frac{1}{\int_{0}^{\mu(\tau^{-n}(A))}h(t)dt}\)=\infty.$$
(c)$\implies$(d):-
Suppose (c) will be holds. Then $\displaystyle\lim_{n\to \infty}\sup\varphi^{-1}\(\frac{1}{\int_{0}^{\mu(\tau^{-n}(A))}h(t)dt}\)=\infty.$\\
As $\varphi$ satisfies the $\Delta_{2}$ condition and by using the Lemma \ref{lem1}, we have
$$\displaystyle\lim_{n \to \infty} \sup \frac{1}{\int_{0}^{\mu(\tau^{-n}(A))}h(t)dt}=\infty$$
and consequently,
$$\displaystyle\lim_{n \to \infty} \inf \int_{0}^{\mu(\tau^{-n}(A))}h(t)dt=0.$$
Taking here $\mu(X)<\infty$ and using the fact that $\varphi$ is continuous and satisfies the $\Delta_{2}$ condition, we have
$$\displaystyle\lim_{n \to \infty} \inf \int_{0}^{\mu(\tau^{n}(A))}h(t)dt=0.$$
By Lemma \ref{lem1}, we get 
$$\displaystyle\lim_{n \to \infty}\sup \varphi^{-1}\(\frac{1}{\int_{0}^{\mu(\tau^{n}(A))}h(t)dt}\)=\infty.$$
(d) $\implies$ (e) follows on the similar line in (c) $\implies$ (d). Thus the result holds. 
\end{proof}

In the following, we consider the space $L^{\varphi}(X,\mathbb{A},\mu)\cap L^{\infty}(X,\mathbb{A},\mu)$ with the norm defined as
$$||g||_{\varphi \cap \infty} = max \{||g||_{\varphi},||g||_{\infty} \}.$$

\begin{theorem}
Assume that $(X,\mathbb{A},\mu)$ be non-atomic infinite measure space and $\tau$ be non-singular measurable transformation. Let $\varphi$ be an Orlicz function which satisfies the $\Delta_{2}$-condition for all $s>0$. Then $C_{\tau}$ on $L^{\varphi}(X,\mathbb{A},\mu)\cap L^{\infty}(X,\mathbb{A},\mu)$ is not Li-Yorke chaotic.
\end{theorem}

\begin{proof}
Suppose on the contrary that $C_{\tau}$ is Li-Yorke chaotic. Then by Theorem \ref{prop}, it admits an irregular vector $g \in L^{\varphi}(X,\mathbb{A},\mu)\cap L^{\infty}(X,\mathbb{A},\mu)$. Let $\{n_{j}\}$ be an increasing sequence of positive integers such that $C_{\tau}^{n_{j}}g \to 0$ in $L^{\varphi}(X,\mathbb{A},\mu)$. Then it converges in $L^{\varphi}$ norm also.\\
Therefore,
\begin{eqnarray*}
||C_{\tau}^{n_{j}}g||_{\varphi \cap \infty} &=& max\{||C_{\tau}^{n_{j}}g||_{\varphi},||C_{\tau}^{n_{j}}g||_{\infty}\}\\
                                        &\leq& max\{0,1\} \\
																				&\leq& 1.
\end{eqnarray*}
which contradicts the fact that $g$ is an irregular vector for $C_{\tau}$.
\end{proof}

\end{document}